\documentclass[12pt]{article}

\usepackage{amsmath, amsthm, graphicx, color, verbatim, amsfonts, wasysym, 
braket, subfigure, fullpage}
\usepackage[latin1]{inputenc}

\DeclareMathOperator{\Tr}{Tr}
\newcommand{\abcd}{\begin{pmatrix}a&b\\c&d\end{pmatrix}}
\newcommand{\bs}{{\boldsymbol\sigma}}
\newcommand{\bt}{{\boldsymbol\tau}}
\renewcommand{\u}{\mathbf{u}}
\newcommand{\x}{\mathbf{x}}

\newcommand{\occ}{\CIRCLE}
\newcommand{\vac}{\Circle}

\newtheorem{lemma}{Lemma}
\newcommand{\Ref}[1]{(\ref{#1})}
\newcommand{\nn}{\nonumber \\}

\begin{document}

\title{Upper bounds on the growth rates of hard squares and related models via corner transfer matrices}
\author{Yao-ban Chan$^1$\thanks{E-mail: \texttt{yaoban.chan@uq.edu.au}}~and Andrew Rechnitzer$^2$ \\ \\ \normalsize $^1$School of Mathematics and Physics, The University of Queensland, Brisbane, Australia \\ \normalsize $^2$Department of Mathematics, University of British Columbia, Vancouver, Canada}

\maketitle

\begin{abstract}
  \textbf{Abstract.}
  We study the growth rate of the hard squares lattice gas, equivalent to the number of independent sets on the square lattice, and two related models --- non-attacking kings and read-write isolated memory. We use an assortment of techniques from combinatorics, statistical mechanics and linear algebra to prove upper bounds on these growth rates. We start from Calkin and Wilf's transfer matrix eigenvalue bound, then bound that with the Collatz-Wielandt formula from linear algebra. To obtain an approximate eigenvector, we use an ansatz from Baxter's corner transfer matrix formalism, optimised with Nishino and Okunishi's corner transfer matrix renormalisation group method. This results in an upper bound algorithm which no longer requires exponential memory and so is much faster to calculate than a direct evaluation of the Calkin-Wilf bound. Furthermore, it is extremely parallelisable and so allows us to make dramatic improvements to the previous best known upper bounds. In all cases we reduce the gap between upper and lower bounds by 4-6 orders of magnitude.

\end{abstract}

\section{Introduction}
We study the growth rate of the statistical mechanical model of a hard square lattice 
gas. In this model, the vertices of a two-dimensional square lattice are 
either occupied by a gas particle or vacant, denoted by $\occ$ and $\vac$ respectively. 
However, no two occupied vertices may be immediate neighbours. We can 
think of each occupied vertex as being covered by a square (rotated $45^\circ$ and with 
side-length $\sqrt{2}$). The constraint forbids these squares from 
overlapping, giving rise to the name hard squares. 

In one dimension, the hard square lattice gas can be used as a simple model of data 
storage on a magnetic tape. As the tape is read from left to right, a vacant vertex 
indicates that the stored field is unchanged, while an occupied vertex indicates that the 
field flips. To avoid potential \emph{intersymbol interference}, we forbid field flips from occurring in close succession --- hence 
$\occ\occ$ is forbidden \cite{Lind1995}. Because of this connection with magnetic fields, the vertices and their states are often referred to as \emph{spins}.

The storage capacity of a length of tape is dependent on the number of legal configurations of the spins it contains.
Accordingly, we define the \emph{partition function}, which is simply the total number of 
possible configurations: 
\begin{align}
 Z_N = \text{\# of legal configurations of $N$ spins}.
\end{align}
The partition function typically grows exponentially with respect to the number of 
vertices, and so a more appropriate measure for the number of possible configurations is 
the \emph{partition function per site} or \emph{growth rate}
\begin{align}
\kappa = \lim_{N \rightarrow \infty} Z_N^{1/N}.
\end{align}
The \emph{capacity} of the system is the number of independent bits of storage per vertex in the system. This is directly related to $\kappa$ by
\begin{align}
\text{capacity} = \log_2 \kappa.
\end{align}

While growth rates for one-dimensional systems can be readily computed as the dominant eigenvalue of a finite transfer matrix\footnote{This implies that $\kappa$ for a 1d system is a Perron number. Lind \cite{Lind1983} proved that any Perron number is also the growth rate of some similar 1d system.}, the corresponding problem for two-dimensional systems is typically very hard. Despite considerable effort, exact growth rates are only known for a very small number of systems\footnote{For example, the fully-packed dimer model \cite{Kasteleyn1961, Temperley1961}, 3-colourings of the square grid \cite{Lieb1967}, and the ``odd'' model \cite{louidor2010}.}. Indeed, it has been shown by Berger \cite{Berger1966} that there exist systems for which it is undecidable if there exist any valid configurations at all. Because of this, computing growth rates has become something of a numerical challenge and has been studied extensively for many different models under numerous guises. 

\begin{figure}
 \begin{center}
  \includegraphics[width=0.65\textwidth]{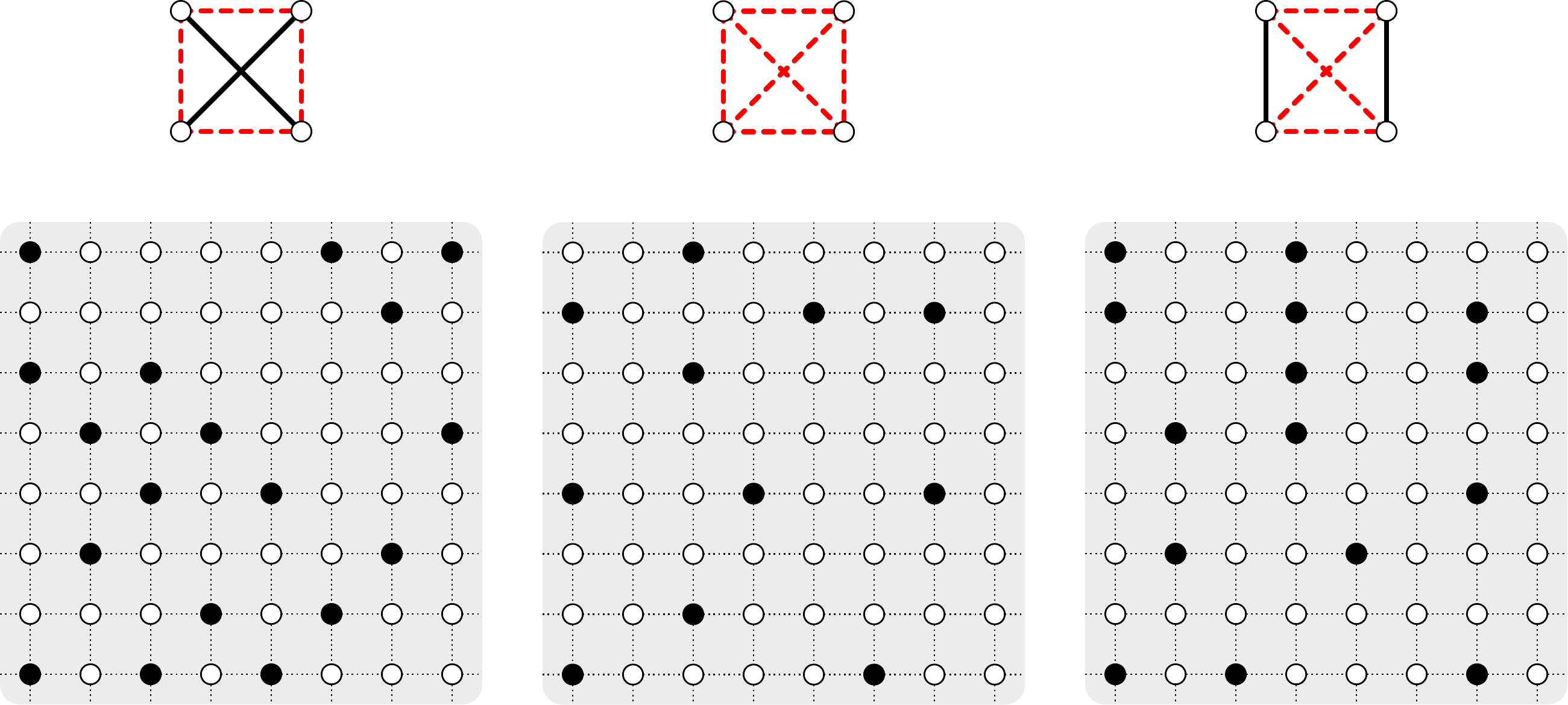}
 \end{center}
  \caption{The models that we study in this paper. The above diagrams show the restrictions on adjacent $\occ$ spins --- dashed red lines are disallowed and solid black lines are allowed. The lower diagrams show example configurations. From left to right: 
hard squares, non-attacking kings (NAK) and read-write isolated memory (RWIM).}
\label{fig exc models}
\end{figure}

One of the most-studied models in this area is the aforementioned hard squares model \cite{CalkinWilf, 
louidor2010, Friedland2010}. This model is well-studied in statistical 
mechanics\footnote{It is a close relative of the famous hard hexagons model 
which was solved exactly by Baxter \cite{Hardhex}.}, not just for the combinatorics of 
the number of configurations, but also for its macroscopic behaviour as the relative weighting of 
$\occ$ and $\vac$ spins changes. Obtaining very precise upper bounds for the growth rate of this 
model is the main focus of this paper.

We also consider two related models which forbid certain local configurations 
of $\occ$ spins (see Figure~\ref{fig exc models}). The \emph{non-attacking kings} model 
forbids horizontal, vertical or diagonal adjacency of $\occ$ spins --- they can be 
considered as kings on a chessboard which cannot be placed in such a way that one attacks 
another. The \emph{read-write isolated memory} model \cite{RWIM1, RWIM2} forbids two 
horizontally or diagonally adjacent $\occ$ spins (vertical adjacency is allowed). If we 
consider the two-dimensional array of spins as a horizontal line of spins evolving in time 
(each row separated by 1 time unit), then the RWIM constraint can be viewed as forbidding us 
from altering two adjacent spins in one time step. 

It is not at all clear that tractable closed-form expressions for $\kappa$ exist for these models, 
and none are known. There is a significant body of work on finding rigorous bounds on 
$\kappa$, most of which are based on the analysis of transfer matrices 
(see Figure~\ref{fig tm1d2d}). Rather than writing down a transfer matrix for an 
infinite two-dimensional system --- an obviously impossible task --- we consider the system on an infinite strip of finite width, say $w$ vertices.
\begin{figure}
 \begin{center}
  \includegraphics[width=0.6\textwidth]{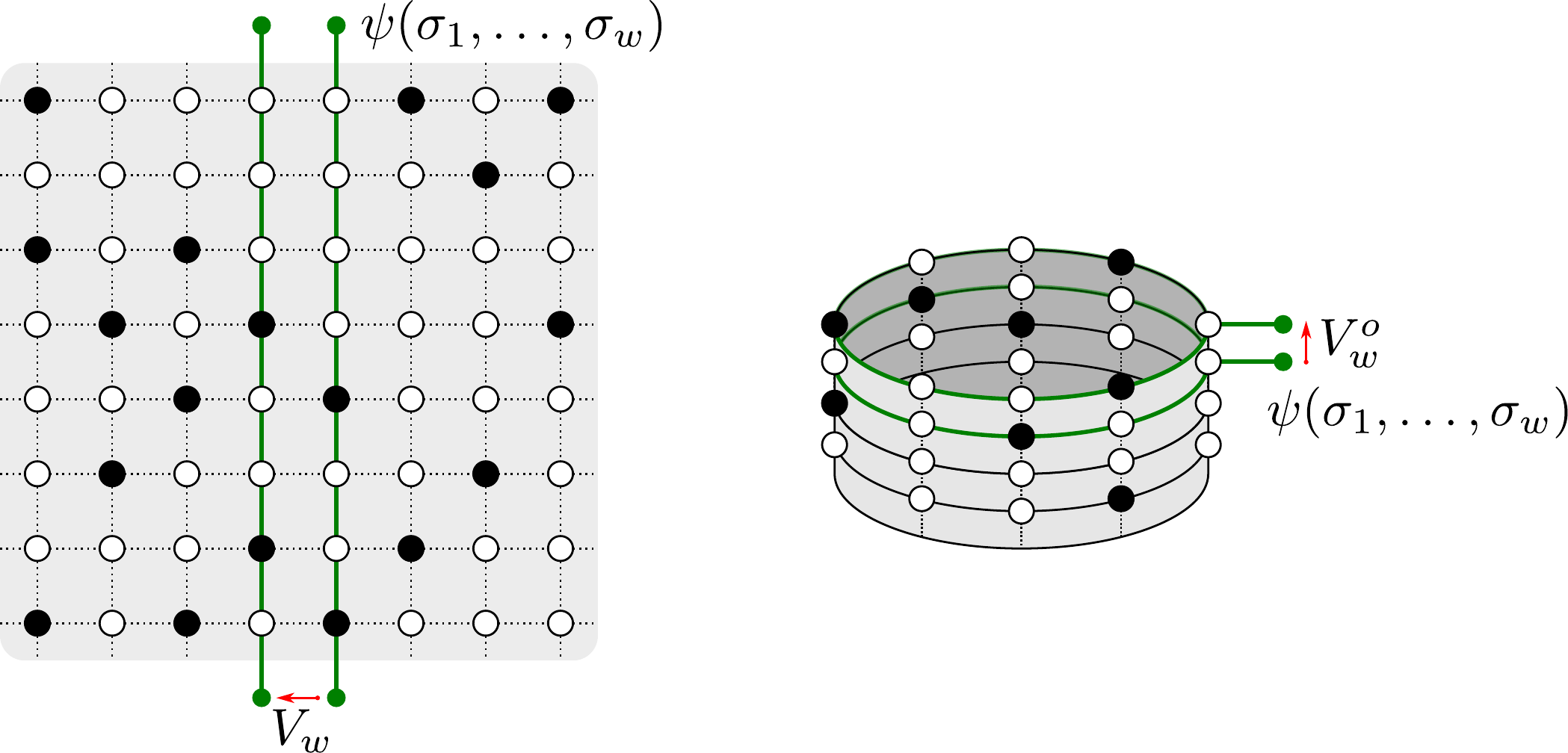}
 \end{center}
 \caption{The transfer matrix $V_w$ extends a system of $N$ columns to a system of $N+1$ columns by using the state of the $N$th column $(\sigma_1,\ldots,\sigma_w)$. The transfer matrix $V^o_w$ constructs the same system on a cylinder.}
\label{fig tm1d2d}
\end{figure}
Let $V_w$ be the column transfer matrix associated with this strip, and $\Lambda(w)$ its dominant 
eigenvalue. Then 
\begin{align}
  \label{eqn kappa defn}
  \kappa &= \lim_{w \to \infty} \Lambda(w)^{1/w}.
\end{align}
Lower bounds for $\kappa$ can then be computed using a clever formula based on Rayleigh 
quotients due to Engel \cite{Engel1990}, and Calkin and Wilf \cite{CalkinWilf} (we refer 
the reader to the latter paper for its proof): 
\begin{align}
  \kappa^p & \geq \frac{\Lambda(p+2q)}{\Lambda(2q)} & \text{with } p,q >0.
  \label{eqn lower}
\end{align}
Calkin and Wilf also proved the following upper bound on $\kappa$, and we again refer the reader to their paper for its proof:
\begin{align}
  \kappa^{2p} & \leq \Lambda^o(2p) & \text{with } p>0,
  \label{eqn upper}
\end{align}
where $\Lambda^o(w)$ is the dominant eigenvalue of a related matrix, $V^o_w$, the 
transfer matrix for the same system but with cylindrical boundary conditions (see 
Figure~\ref{fig tm1d2d}).

Almost all works which compute bounds on the growth rates of the hard squares and related models use the above two inequalities (two recent exceptions being \cite{Gamarnik2009} and previous work by the authors \cite{lowerbounds}, both of which use methods from statistical physics). However, the practical application of these inequalities is quickly hampered by the exponential growth of the size of transfer matrix with respect to strip width.\footnote{For the three models we consider, the matrices grow as the Fibonacci numbers.} There are creative methods to avoid storing the full matrix, perhaps the most successful being a matrix compression method due to Lundow and Markstr\"om \cite{Lundow}. Together with Friedland \cite{Friedland2010}, they used this method to compute $\Lambda(28)$ and $\Lambda^o(36)$, which in turn allowed them to exactly determine the first 15 digits of $\kappa$ for the hard squares model:
\begin{align}
  \kappa &=  1.503\,048\,082\,475\,33\dots
\end{align}
While this compression method greatly decreases the time and memory needed, the 
requirements still grow exponentially with strip width.

In this work, we compute upper bounds on the growth rates of the models. We start with equation~\Ref{eqn upper}, but 
rather than computing $\Lambda^o(w)$ exactly, we compute an upper bound on it using an 
approximation of the dominant eigenvector of $V^o_w$ and the Collatz-Wielandt formula 
\cite{collatz,wielandt} (which we include as Lemma~\ref{lem:collatz}). To form the approximate 
eigenvector, we use corner transfer matrix formalism. This is a very powerful 
approach developed in statistical mechanics by Baxter \cite{BaxterDimers,CTM:1,CTM:Hsq} as a way to estimate the partition function of various 
lattice models, both numerically and via series expansions \cite{CTM:Hsq,CTM:Hardsq,chan2013series}. 

The CTM formalism allows us to express each component of the approximate eigenvector as 
the trace of a product of auxiliary matrices. This means that we are no longer required to 
store either the transfer matrix or vector; we can compute the vector component by component as the 
Collatz-Wielandt formula requires. To optimise the choice of vector (and the associated 
auxiliary matrices), we use an extension of CTM known as the corner transfer matrix 
renormalisation group (CTMRG) method, developed by Nishino and Okunishi \cite{CTM:RG1,CTM:RG2}.

\section{Upper bounds on upper bounds}

We start by restating Calkin and Wilf's upper bound on the growth rate $\kappa$.
\begin{lemma}[\cite{CalkinWilf}]
\label{lem:calkinwilf}
  Let $V^o_w$ be the column transfer matrix for a system of width $w$ with cylindrical 
boundary conditions, and let $\Lambda^o(w)$ be its dominant eigenvalue. Then for any even integer $m>0$,
\begin{equation}
  \kappa^{m} \leq \Lambda^o(m). \label{eq:calkinwilf}
\end{equation}
\end{lemma}

As noted in the introduction, almost all upper bounds in the literature are derived 
using Lemma \ref{lem:calkinwilf} by exact calculation of the dominant eigenvalue of $V^o_w$. As such they are restricted, both in time and memory, by the exponential growth of the transfer matrix size. We also start with this bound, but do not compute the eigenvalue exactly.
Instead we find upper bounds for $\Lambda^o(m)$ using the Collatz-Wielandt 
formula \cite{collatz,wielandt}.
\begin{lemma}[Collatz-Wielandt formula]
  \label{lem:collatz}
  Let $A$ be an irreducible square matrix with non-negative entries. Then 
for any vector $\x > 0$, the largest eigenvalue of $A$ 
(denoted $\lambda$) is real and positive and is bounded by
\begin{equation}
  \min_i \frac{(A\x)_i}{x_i} \leq \lambda \leq \max_i \frac{(A\x)_i}{x_i}. \label{eq:collatz}
\end{equation}
\end{lemma}
\begin{proof}
 Let $\u$ be the left eigenvector of $A$ corresponding to $\lambda$. By 
the Perron-Frobenius theorem, $\lambda$ is real and positive and $\u$ has strictly 
positive  entries. Then for any $\x > 0$,
  \begin{align}
    \Braket{\u | A | \x} & = \lambda \Braket{\u | \x} 
  \end{align}
  and so
  \begin{align}
    \Braket{\u | A\x - \lambda \x } & = \sum_{i=1}^n u_i x_i \left( 
\frac{(A\x)_i}{x_i} - \lambda \right) = 0.
\end{align}
Some terms in the sum must be non-negative and some non-positive (although they may all be 0 if $\x\propto\u$), and therefore the maximum of the summands is non-negative and the minimum is non-positive. Since $\u, \x > 0$, the result follows.

\end{proof}

We use this lemma to find an upper bound on $\Lambda^o(m)$. In order to calculate a tight bound, we need to get as close as possible to the dominant eigenvector of $V^o_m$. We do this using Baxter's corner transfer matrix ansatz, which specifies to choose a vector $\psi$ of the form
\begin{align}
  \psi_{\bs} = \Tr \left[ F(\sigma_1,\sigma_2) F(\sigma_2,\sigma_3) \dots F(\sigma_m \sigma_1)\right], \label{eq:psi}
\end{align}
where the index $\bs = (\sigma_1,\dots,\sigma_m) \in \{\occ,\vac\}^m$ runs over all legal configurations of one row of the cylinder, and 
$\{F(a,b)\}$ is a set of four matrices of size $n \times n$, indexed by two spin values $a$ and $b$. Here $n$ is an arbitrary positive integer which need not be related to $m$. The $F$ matrices are calculated from corner transfer matrix methods; the process is quite involved, and we leave its description until the next section. We can interpret this ansatz pictorially by thinking of the $F$ matrices as half-row transfer matrices which build up the state $\psi_\bs$ one row at a time --- see Figure~\ref{fig:VandF}.

\begin{figure}
  \begin{center}
    \includegraphics[width=0.9\textwidth]{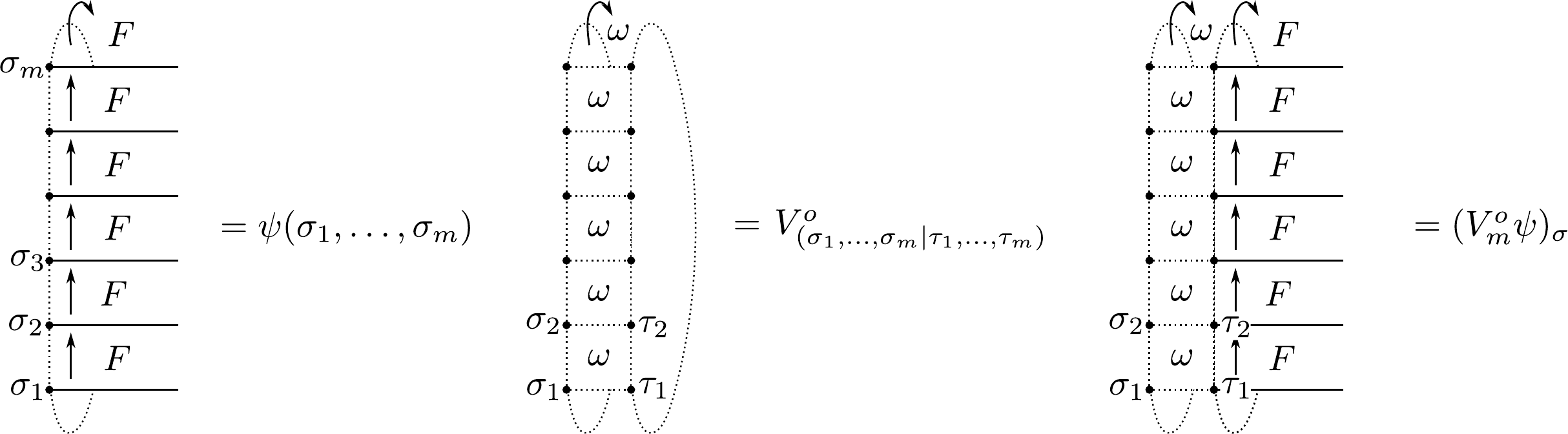}
  \end{center}
  \caption{Pictorial interpretation of \Ref{eq:psi},~\Ref{eqn V in omega} and~\Ref{eqn Vpsi}.}
  \label{fig:VandF}
\end{figure}

In order to use \Ref{eq:collatz}, we must also be able to compute $\left( V^o_m\psi \right)_\bs$, and to do this we define the \emph{face weight} $\omega$. The weight of a face (a single cell of the lattice) is 1 when the spins around it form a legal configuration and 0 otherwise. The element of $V^o_m$ which maps a column in state $\bt$ to a column in state $\bs$ is then given by the product
\begin{align}
  \label{eqn V in omega}
  (V^o_m)_{\bs,\bt} &=
\prod_{i=1}^m \omega\begin{pmatrix} \sigma_{i+1} & \tau_{i+1} \\ 
\sigma_i & \tau_i \end{pmatrix}, & \text{where } 
\begin{array}{rl}
\sigma_{m+1}&\equiv\sigma_1, \text{ and}\\ 
\tau_{m+1}&\equiv\tau_1.
\end{array}
\end{align}
This is shown pictorially in Figure~\ref{fig:VandF}. We can then write the action of $V^o_m$ on $\psi$ as 
\begin{align}
  \left( V^o_m\psi \right)_\bs & = \sum_\bt \left( V^o_m \right)_{\bs,\bt} \psi_\bt = \sum_\bt \prod_{i=1}^m \omega\begin{pmatrix} \sigma_{i+1} & \tau_{i+1} \\ \sigma_i & 
\tau_i \end{pmatrix} \Tr \left[ F(\tau_1,\tau_2) F(\tau_2,\tau_3) \dots F(\tau_m 
\tau_1)\right] \nn
    & = \Tr \sum_\bt \prod_{i=1}^m \omega\begin{pmatrix} \sigma_{i+1} & \tau_{i+1} \\ 
\sigma_i & \tau_i \end{pmatrix} F(\tau_i, \tau_{i+1}) = \Tr \left[ F_l(\sigma_1,\sigma_2) F_l(\sigma_2,\sigma_3) \dots F_l(\sigma_m 
\sigma_1)\right], \label{eqn Vpsi}
\end{align}
where $F_l$ is a $2n \times 2n$ matrix defined by the block matrix equation (shown pictorially in Figure~\ref{fig:enlarge})
\begin{align}
  \left.F_l(c,a)\right|_{d,b} & = \omega \abcd F(d,b). \label{eq:fl}
\end{align}

We are now able to combine the above expression with \Ref{eq:calkinwilf} and \Ref{eq:collatz} to derive
\begin{align}
  \kappa^m \leq \Lambda^o(m) & \leq \max_{\bs} \frac{\left( V^o_m\psi \right)_\bs}{\psi_{\bs}} 
  = \max_{\bs} \left\{ \frac{ \Tr \left[ F_l(\sigma_1,\sigma_2) F_l(\sigma_2,\sigma_3) 
\dots F_l(\sigma_m \sigma_1)\right] }{ \Tr \left[ F(\sigma_1,\sigma_2) 
F(\sigma_2,\sigma_3) \dots F(\sigma_m \sigma_1)\right] } \right\}. \label{eq:upper}
\end{align}
This is the upper bound that we use for $\kappa$, and it is valid for any $m$, $n$, and $F$, as long as we satisfy the following conditions:
\begin{itemize}
  \item $V^o_m$ is non-negative. In fact $V^o_m$ is a 0-1 matrix for the considered models, so this follows immediately.
  \item $V^o_m$ is irreducible. This is equivalent to showing that every set of states 
$\bs$ on a cut can be reached from any other. This is easy to show for our 
models, as there is no restriction on adjacency to $\vac$ spins --- so any $\bs$ can be 
adjacent to $\{\vac,\vac,\ldots,\vac\}$, and thus can reach any other set of states.
  \item $\psi_\bs$ is positive. This does not immediately follow from (\ref{eq:psi}). We 
must show this by explicitly computing (\ref{eq:psi}) for every $\bs$ and verifying that 
it is positive. However, this does not result in any extra work since we already need these values to compute (\ref{eq:upper}).
\end{itemize}

Putting all of this together, we find an upper bound as follows:
\begin{enumerate}
  \item Select $m, n > 0$ with $m$ even.
  \item Calculate a set of $n \times n$ matrices $\{F(a,b)\}$ as specified in the next section.
  \item For all possible cut states $\bs$ of $m$ spins: \label{step:calc}
    \begin{enumerate}
	 \item Verify $\psi_\bs > 0$, where $\psi_\bs$ is given by (\ref{eq:psi}).
	 \item Calculate $\left(\frac{\left( V^o_m\psi \right)_\bs}{\psi_{\bs}}\right)^{1/m}$. \label{step:bound}
    \end{enumerate}
  \item The upper bound for $\kappa$ is the maximum of all values calculated in step \ref{step:bound}.
\end{enumerate}
This method uses very little memory, because we do not need to keep the entire $\psi$ vector in memory. Each component can be computed independently as it is needed from the $F$ and $F_l$ matrices. These matrices are tiny compared to $\psi$ and $V^o_m$, so this is a very modest requirement. Furthermore, the calculation in step~\ref{step:calc} for any particular component does not depend on any other component, so can be done in parallel (and ratios compared afterwards). This step is by far the most time-consuming part of the process --- the number of traces required grows exponentially with $m$. Hence parallelisation creates a huge real-time saving. Details of our implementation and some optimisations are given in Section~\ref{sec results} below.

\section{Approximate eigenvectors from corner transfer matrices}

So far, we have not described how to calculate the $F$ matrices, and this is clearly critical in order to obtain a good upper bound. We need to choose the matrices in such a way that the approximate eigenvector $\psi$ given by \Ref{eq:psi} is close to the true eigenvector of $V^o_m$. We accomplish this through the use of corner transfer matrix formalism, which we briefly outline here; we direct the reader to a more detailed discussion in~\cite{lowerbounds}.

The expression \Ref{eq:psi} is the starting point of Baxter's corner transfer matrix formalism. It can be represented pictorially by considering $F(a,b)$ in the limit $n \rightarrow \infty$ as a half-row transfer matrix; it takes a half-infinite row of spins ending with spin $a$, and transfers it along one row to a half-infinite row of spins ending with spin $b$. Using this representation, we can see that the expression for $\psi$ represents a half-plane partition function (normalised in some way), which is unchanged under the action of $V^o_m$ save for a constant factor --- \textit{i.e.}, it is an eigenvector of $V^o_m$. This representation is shown in Figure \ref{fig:VandF}.

These infinite-dimensional half-row transfer matrices can be shown \cite{CTM:1,lowerbounds} to satisfy the \emph{CTM equations}
\begin{align}
  \xi A^2(a) & = \sum_b F(a,b) A^2(b) F(b,a), \\
  \eta A(a) F(a,b) A(b) & = \sum_{c,d} \omega\abcd F(a,c) A(c) F(c,d) A(d) F(d,b),
\end{align}
where $A(a)$ are matrices of the same size as $F$, indexed by a single spin value. The $A$ matrices are the eponymous \emph{corner transfer matrices}; they transfer a half-infinite row of spins ending with spin $a$ around an angle of $\pi/2$, sweeping out a quarter of the lattice.
The CTM equations can then be interpreted as equating half-plane transfer matrices (see Figure \ref{fig:XandY}). The matrices on either side of the equality differ only by an extra row, which results in the scalar multipliers $\xi$ and $\eta$.

\begin{figure}
  \begin{center}
    \includegraphics[width=\textwidth]{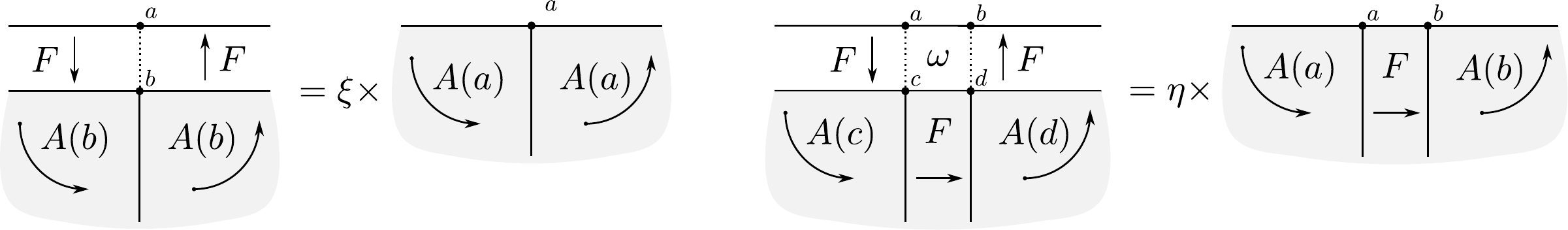}
  \end{center}
  \caption{Pictorial interpretation of the CTM equations.}
  \label{fig:XandY}
\end{figure}

In the infinite-dimensional limit, the CTM equations directly yield the solution of the model via $\kappa = \eta/\xi$. We can approximate the true solution by taking $F$ matrices of finite size $n \times n$. The CTM equations can be solved for finite-dimensional matrices, and we use these finite sized solutions for our $F$ matrices in our bound finding procedure.

Solving the CTM equations at finite size is quite non-trivial. For this purpose, we use the 
corner transfer matrix renormalisation group method developed by Nishino and Okunishi 
\cite{CTM:RG1,CTM:RG2}. In this method, we start with some trial values for $A(a)$ and $F(a,b)$ which are then ``polished'' into solutions.
We do this by expanding $A$ and $F$ into the $2n \times 2n$ matrices $A_l$ and $F_l$, using the block matrix equations \Ref{eq:fl} and
\begin{align}
  \left.A_l(c)\right|_{d,a} & = \sum_b \omega\abcd F(d,b) A(b) F(b,a).
\end{align}
These can be interpreted as ``larger'' versions of the corner and half-row transfer matrices (see Figure 
\ref{fig:enlarge}).

\begin{figure}
  \begin{center}
    \includegraphics[width=0.9\textwidth]{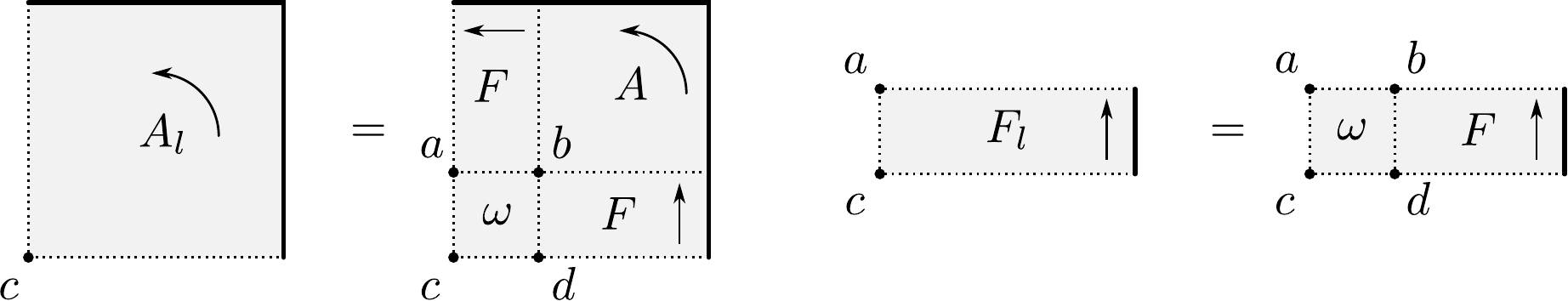}
  \end{center}
  \caption{Pictorial interpretation of $A_l$ and $F_l$.}
  \label{fig:enlarge}
\end{figure}

We then reduce the large matrices to generate new iterations of $A$ and $F$. To do this, 
we diagonalise $A_l(a)$, producing orthogonal diagonalising matrices $P(a)$. We take the 
transformations
\begin{align}
  A_l(a) & \rightarrow P^T(a) A_l(a) P(a), &  F_l(a,b) & \rightarrow P^T(a) F_l(a,b) P(b),
\end{align}
which leave the CTM equations invariant. We then truncate the matrices to the original $n 
\times n$ size by keeping the $n$ largest eigenvalues of $A_l$ and 
performing a consistent truncation on $F_l$. This has the effect of making the matrices intuitively ``cover as large an area as possible'', so that they are close to the infinite-size solution of the CTM equations.

The expanding and reducing procedures are iterated, and it is observed that the matrices eventually converge 
to the finite-size solution of the CTM equations. The initial values for the $A$ and $F$ 
matrices can be taken to be of some small size, typically $1 \times 1$ or $2 \times 2$. 
They can then be ``built up'' to the desired $n \times n$ size by sometimes keeping extra 
eigenvalues at the reduction step, resulting in larger matrices until the desired size is 
reached. We again direct the reader to \cite{lowerbounds} for a more detailed description of this process.

\section{Results}
\label{sec results}
The method was implemented in 2 distinct steps: computing the $F$ matrices, and then computing the trace ratios for all legal states $\bs$. Both parts were implemented in \texttt{C++} using the \texttt{Eigen} numerical linear algebra library \cite{eigenweb}. We used \texttt{Eigen} since it readily supports multiple precision computation through the \texttt{MPFR} library \cite{MPFR}. High precision is necessary in the first step because the eigenvalues of corner transfer matrices range over many orders of magnitude, and in the second step so that the trace ratios are also of high precision --- though less precision is needed. The first step requires only modest computing resources, and all $F$ matrices were generated within a few hours on a modest Linux laptop using only a few megabytes of memory.

The second step --- computing the trace ratios --- can be implemented quite na\"ively and still give good results. However we use a number of simple and very effective optimisations. Firstly, traces are invariant under cyclic permutations, so we only have to consider legal cut states modulo cyclic permutations for $\bs$. Traces are also invariant under transposes, and the $F$ matrices satisfy $F(a,b) = F^T(b,a)$, so we can also disregard reflections. Thus we consider bracelets which are legal for the system (generated using the methods in \cite{Ruskey2000}). This reduces the number of computations by a factor of approximately $2m$. Additionally, we can gain a little more speed by using similarity transforms to diagonalise one of the $F$ matrices. Finally, we split the set of bracelets into batches of equal size across about 300 cores on the WestGrid computing cluster. Each batch (with $n=30, m=50$) took approximately 1 day of CPU time. Note that final upper bounds were computed at high precision and then 
rounded up.

Our results are given in Table \ref{tab:results}. For each of the three models, we computed the bound using $F$ matrices of size $30\times 30$ on cylindrical systems of circumference $50$ vertices. In this table, we also indicate the previous best upper bounds. In Table~\ref{tab:known digits}, we combine these new upper bounds with the best known lower bounds to show the digits of the growth rates which we now know rigorously. We extend the number of rigorously known digits from previous work by between 4 and 6 digits for our models.

We note here that the use of software floating point precision for the trace-ratio computations made our code about 50-100 times slower than if we had been able to use hardware floating point. Unfortunately most commercially available processors, including those we had access to, support 80-bit double-extended precision rather than true 128-bit or higher precision. If such high precision hardware floating point were available, then we could have increased $m$ by about 8, which would perhaps give an extra 2 or 3 digits exactly. The previous best upper bounds \cite{Friedland2010,louidor2010} used machine floating point.

\begin{table}
  \centering
  \begin{tabular}{|l|ll|l|}
  \hline
    Model & $m$ & $n$ & Upper bound \\
    \hline
    Hard squares & 50 & 30 & \underline{1.503 048 082 475 33}2 264 519 \\
    NAK & 50 & 30 & \underline{1.342 643 95}1 124 602 238 \\
    RWIM & 50 & 30 & \underline{1.448 9}57 372 102 \\
\hline 
  \end{tabular}
  \caption{Upper bounds on the growth rates of hard squares, NAK and RWIM. The previous best known bounds \cite{Friedland2010,louidor2010} are underlined.}
  \label{tab:results}
\end{table}

\begin{table}
  \centering
  \begin{tabular}{|l|p{10cm}|l|}
\hline
    Model & Lower bounds \cite{lowerbounds} & Exact digits \\
    \hline
    Hard squares & \underline{1.503 048 082 475 332 264} 322 066 329 475 553 689 385 781 038 610 305 062 028 101 735 933 850 396 923 440 380 463 299 47 & 19 \\
\hline
    NAK &  \underline{1.342 643 951 124 60}1 297 851 730 161 875 740 395 719 438 196 938 393 943 434 885 455 0 & 15 \\
\hline
    RWIM & \underline{1.448 957 37}1 775 608 489 872 231 406 108 136 686 434 371 & 9 \\
\hline
  \end{tabular}
  \caption{Lower bounds on the growth rates of hard squares, NAK and RWIM from \cite{lowerbounds}; combining those lower bounds with the new upper bounds, the underlined digits are known rigorously. Note that the lower bounds agree the best estimates \cite{lowerbounds} in all but their last couple of digits.}
  \label{tab:known digits}
\end{table}

Due to a lack of rigorous results on the theoretical convergence of the CTMRG, it is difficult to predict how good our bounds will be for given $m$ and $n$. Thus we must observe their empirical behaviour for various $m$ and $n$. We first note that $m$ is the computational bottleneck parameter, as the time required is exponential in $m$ (while the time and memory requirements depend polynomially on $n$). Thus we wish to increase $m$ to the limit of our computational power, and then select an optimal $n$ for that value of $m$. We show the results in the context of the hard squares model; other models behave in much the same way.

For any fixed $m$, the bound given by (\ref{eq:upper}) decreases as $n$ increases, but only up to a small limit, after which it stays level. Therefore, given a value of $m$, we can be assured of calculating the best bound for that $m$ as long as $n$ is sufficently large. This is shown in Figure \ref{fig:finite}.
\begin{figure}
  \begin{center}
    \subfigure[Log error in the upper bounds generated by increasing $n$ for $m = 12,24$.]{ \includegraphics[width=0.465\textwidth]{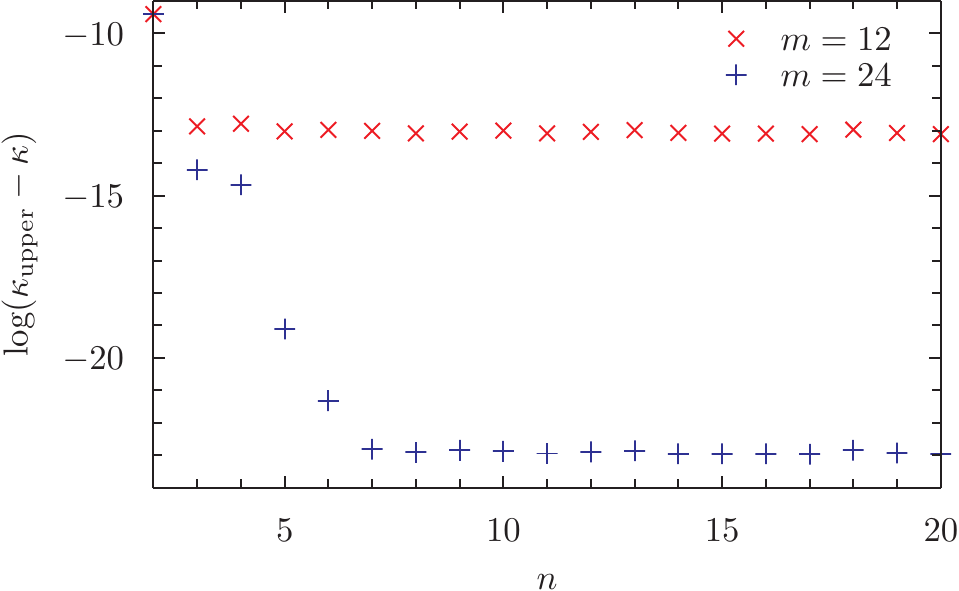} \label{fig:finite} } \hfill
    \subfigure[Optimal $n$ for various $m$. There appears to be a linear relationship with $m$.]{ \includegraphics[width=0.435\textwidth]{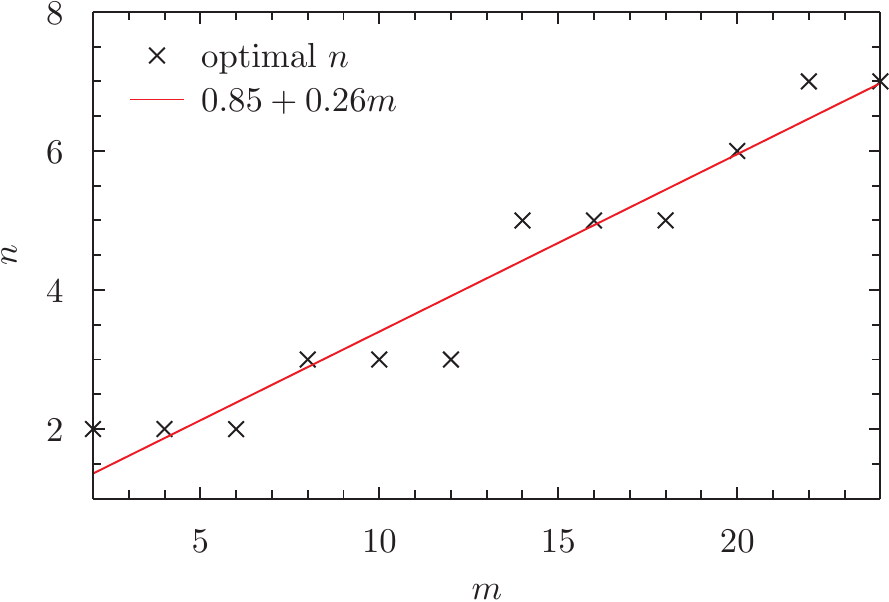} \label{fig:optimal} }
  \end{center}
  \caption{Choosing an $n$ value for a given $m$.}
\end{figure}
The optimal $n$ (\textit{i.e.}, the value of $n$ after which the bound does not decrease further) for any given $m$ appears to follow a linear relationship with $m$. In Figure \ref{fig:optimal}, we show this value for small $m$. For our final bound computation, we were able to run the method for $m = 50$. We estimate by extrapolating the fitted linear relationship that for this value of $m$, the optimal value of $n$ is approximately 14. We have taken $n = 30$, and so are confident that this value is sufficient to achieve the best bound for our value of $m$.

Next, we look at the performance of the bound in relation to $m$. Because our bound is itself an upper bound on $\Lambda^o(m)^{1/m}$, we know that it will be inferior to an explicit evaluation of (\ref{eq:calkinwilf}) for the same $m$. However, for sufficiently large $n$, the difference between (\ref{eq:upper}) and (\ref{eq:calkinwilf}) is very small. This is shown in Figure \ref{fig:error}, where there is almost no discernible difference between the two bounds. Thus we expect that the accuracy of our bound will behave similarly to that of (\ref{eq:calkinwilf}), which is conjectured to decrease to $\kappa$ exponentially with respect to $m$. This behaviour can also be observed in Figure \ref{fig:error}. This means that the algorithm, which is exponential-time in $m$, will also be exponential-time in the number of digits found. This is worse than the observed performance for the lower bound found in \cite{lowerbounds}, although this is not unexpected.

\begin{figure}
  \begin{center}
    \includegraphics[width=0.5\textwidth]{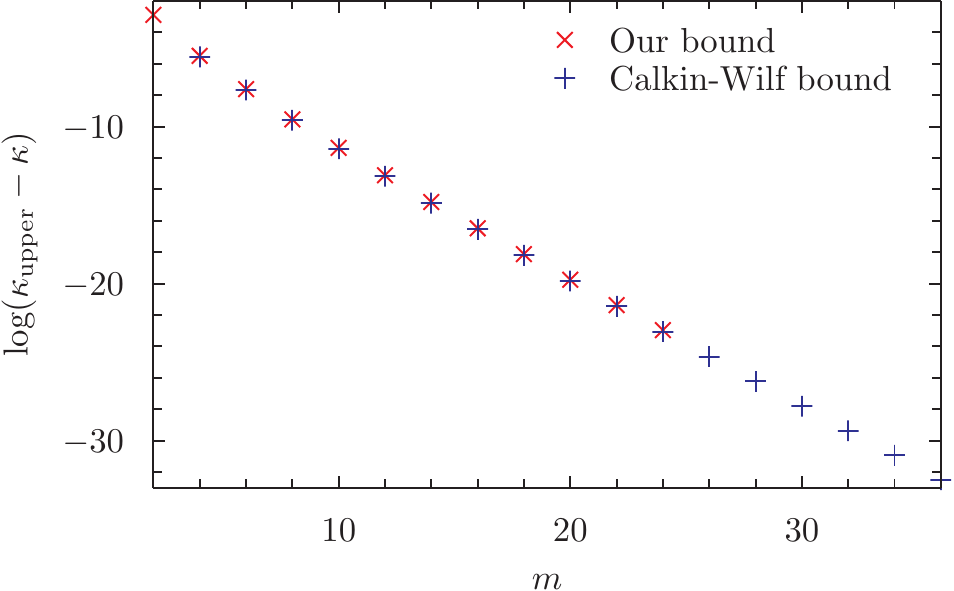}
  \end{center}
  \caption{Log error in the upper bounds generated by (\ref{eq:upper}) (red) and (\ref{eq:calkinwilf}) (blue) for various $m$. Here we used $n = 20$ which always produces the best possible bound for these values of $m$.}
  \label{fig:error}
\end{figure}

As previously mentioned, (\ref{eq:calkinwilf}) is the formula used to generate all previous best known upper bounds. The advantage of our method is that (\ref{eq:upper}) is almost as accurate and much faster to compute, even though both methods are exponential in $m$. For example, in (\ref{eq:calkinwilf}) the eigenvalue $\Lambda^o(36)$ was computed using over 200 gigabytes of memory and 40 days of CPU time (and real time). We were able to push to $m = 50$ using only a few megabytes of memory and about 300 days of CPU time but only 1 day of real time due to parallelisation. Thus we are able to achieve much better bounds.

\section{Conclusion}

In this paper, we have used transfer matrix analysis, the Collatz-Wielandt formula, and corner transfer matrix formalism to derive upper bounds on the growth rates of three lattice models motivated by an information storage problem. Our bounds are a significant improvement on all previously known upper bounds, and together with the corresponding lower bounds from \cite{lowerbounds}, fully determine the growth rates to a generous number of digits.

Our method is not limited to the models we have studied; nearly any lattice model which can be written in terms of local face weights and satisfies some elementary symmetry properties can be analysed in this way. It can be applied directly to colouring models, and it appears possible to exploit symmetries in this case to make it even more efficient. On the other hand, some other models lack irreducible transfer matrices, for example the ``even'' model in \cite{lowerbounds}. We are extending our method to account for these difficulties.

One weakness of our method is that it appears to be exponential-time in the number of digits required, due to the need to evaluate all the ratios in (\ref{eq:upper}) (modulo symmetries). It seems intuitive that much of this is wasted work; since in the limit $n \rightarrow \infty$ it can be shown that the solution of the CTM equations gives the exact eigenvector in (\ref{eq:psi}) and as such, all of the ratios should be very close to each other, and it should be unnecessary to calculate all $2^m$ of them. We observed, for example, that the traces were always positive but we have been unable to prove this. Similarly, we observed that the maximum or minimum trace ratio was frequently given by $\bs=\vac^m$ or $\bs=(\vac\occ)^{m/2}$, but,  again, we have been unable to prove this.  Such results would allow us to compute bounds much more efficiently and we are currently working to prove them.

\section*{Acknowledgements}
The authors would like to thank Ian Enting and Brian Marcus for many helpful and interesting discussions. Additionally they thank WestGrid for providing access to their computer cluster and AR acknowledges financial support from NSERC.

\bibliographystyle{abbrv}
\bibliography{ub_refs}

\end{document}